\documentclass[11pt]{article}
\usepackage{amscd, amsmath, amssymb, amsthm}
\usepackage[all,cmtip]{xy}
\usepackage[pagebackref]{hyperref}

\title{The integral cohomology of the Hilbert scheme
of points on a surface}
\author{Burt Totaro}
\date{  }

\def\Z{\text{\bf Z}}

\def\F{\text{\bf F}}

\def\arrow{\rightarrow}

\def\Ext{\operatorname{Ext}}
\def\Extcal{\mathcal{E}xt}
\def\I{\mathcal{I}}
\def\vir{\operatorname{vir}}
\def\coker{\operatorname{coker}}
\def\K{\widetilde{K}}
\def\homcal{\mathcal{H}om}

\begin{document}
\maketitle
\newtheorem{theorem}{Theorem}[section]
\newtheorem{corollary}[theorem]{Corollary}
\newtheorem{lemma}[theorem]{Lemma}

\theoremstyle{definition}
\newtheorem{definition}[theorem]{Definition}
\newtheorem{example}[theorem]{Example}

\theoremstyle{remark}
\newtheorem{remark}[theorem]{Remark}

The Hilbert scheme $X^{[n]}$ of $n$ points
on a smooth complex surface $X$ is a complex manifold
of dimension $2n$, which can be viewed as a resolution of singularities
of the symmetric product $S^nX$. The rational cohomology
of $X^{[n]}$ is known, but the integral cohomology
is more subtle. Any torsion in cohomology or other invariants
could conceivably be useful for rationality problems.

In this paper, we show
that if $X$ is a smooth complex projective surface
with torsion-free cohomology, then the Hilbert scheme $X^{[n]}$
has torsion-free cohomology for every $n\geq 0$.
(Since we know the Betti numbers of $X^{[n]}$ by G\"ottsche
(stated in Theorem \ref{betti}),
this amounts to an additive calculation of $H^*(X^{[n]},\Z)$.)
We also show that if the integral Chow motive of $X$ is trivial
(a finite direct sum of Tate motives), then the integral
Chow motive of $X^{[n]}$ is trivial for all $n$
(Theorem \ref{motive}).

There are some earlier results in this direction.
When $X$ is the complex projective plane,
Ellingsrud and Str\o mme found an algebraic cell decomposition
of the Hilbert scheme $X^{[n]}$, which implies
that its integral cohomology is torsion-free \cite[Theorem 1.1]{ES}.
Markman showed that the integral cohomology of the Hilbert scheme
$X^{[n]}$
is torsion-free for a smooth projective surface $X$
with a nontrivial Poisson structure,
or equivalently when the anticanonical bundle $-K_X$
has a nonzero section \cite[Theorem 1]{Markman}. That includes
the important case where $X$ is a K3 surface, so that $X^{[n]}$
is hyperk\"ahler. In this paper, we show that the Poisson
assumption can be dropped completely. The fact that $H^*(X,\Z)$
torsion-free implies $H^*(X^{[2]},\Z)$ torsion-free was shown (in fact
for $X$ of any dimension)
in \cite[Theorem 2.2]{TotaroHilbert}.
Finally, for $X$ a smooth projective surface
with first Betti number zero,
Li and Qin gave an explicit basis for $H^*(X^{[n]},\Z)$ modulo torsion
\cite[Theorem 1.2]{LQ}.

Our proofs combine Markman's ideas with the reduced obstruction
theory for nested Hilbert schemes of surfaces
found by Gholampour and Thomas \cite{GT}.

Several related questions remain open. Do the results
of this paper extend to compact complex surfaces, or even
to noncompact complex surfaces? (For the Hilbert square $X^{[2]}$,
the answer is yes, by \cite[Theorem 2.2]{TotaroHilbert}.) Second, say
for a smooth projective surface $X$, is the graded
abelian group $H^*(X^{[n]},\Z)$ determined by the graded
abelian group $H^*(X,\Z)$ when
$H^*(X,\Z)$ has torsion? (We know that the graded vector space
$H^*(X^{[2]},\F_2)$ is not determined by the graded vector 
space $H^*(X,\F_2)$,
by \cite[Example 2.5]{TotaroHilbert}.) Analogously,
is the integral Chow motive of $X^{[n]}$ determined
by that of $X$? Finally, for a complex
manifold $X$ of any dimension, does $H^*(X,\Z)$ torsion-free
imply $H^*(X^{[3]},\Z)$ torsion-free?

I thank Stefan Schreieder for useful discussions.

\section{Betti numbers of the Hilbert scheme}

We recall here the calculation of the Betti numbers
of the Hilbert schemes of points on a surface
\cite[equation (2.1)]{GottscheICM}.
This was proved for smooth projective surfaces by G\"ottsche
and generalized to all smooth complex analytic surfaces
with finite Betti numbers by de Cataldo and Migliorini
\cite[Theorem 5.2.1]{dCManalytic}.

Define the Poincar\'e polynomial
of a space $Y$ by $p(Y,t)=\sum_j b_j(Y)t^j$.

\begin{theorem}
\label{betti}
For a smooth
complex analytic surface $X$ with finite Betti numbers,
the Betti numbers of the Hilbert
schemes $X^{[n]}$ are given by the generating function
$$\sum_{n\geq 0}p(X^{[n]},t)q^n=\prod_{k\geq 1}\prod_{j=0}^4
(1-(-t)^{2k-2+j}q^k)^{(-1)^{j+1}b_j(X)}.$$
\end{theorem}

\section{Gholampour-Thomas's reduced obstruction theory}

Gholampour and Thomas constructed the following
``reduced'' obstruction theory for nested Hilbert schemes
of surfaces
\cite[Theorem 6.3]{GT}. This is easy
when $H^1(X,O)=H^2(X,O)=0$, and in general they show
how to remove the contributions of those two cohomology groups.

I would guess that the same obstruction theory exists
on any complex manifold of dimension 2. If so, then the results
of this paper would extend to compact complex surfaces.
Also, Gholampour and Thomas work with surfaces
over the complex numbers, but their proof works verbatim
over any field.

For natural numbers $n_1\geq n_2$, let $\pi$ be the projection
$$X^{[n_1]}\times X^{[n_2]}\times X\arrow X^{[n_1]}\times X^{[n_2]},$$
with the two universal subschemes $\mathcal{Z}_1,\mathcal{Z}_2$
and their ideal sheaves $\I_1\subset\I_2\subset O_{X^{[n_1]}\times
X^{[n_2]}\times X}$.

\begin{theorem}
\label{nested}
Let $X$ be a smooth geometrically connected
projective surface over a field $k$. For
any $n_1\geq n_2$, the 2-step
nested Hilbert scheme $X^{[n_1,n_2]}$ (of 0-dimensional
subschemes of degree $n_1$
containing a subscheme of degree $n_2$) carries a natural
perfect obstruction theory whose virtual cycle
$$[X^{[n_1,n_2]}]^{\vir}\in CH_{n_1+n_2}(X^{[n_1,n_2]})$$
has pushforward to the Chow groups of $X^{[n_1]}\times X^{[n_2]}$
equal to the Chern class $c_{n_1+n_2}(R\homcal_{\pi}(\I_1,\I_2)[1])$.
\end{theorem}

We only need the case $n_1=n_2$ of Theorem \ref{nested}.
That is:

\begin{corollary}
\label{diagonal}
Let $X$ be a smooth geometrically connected
projective surface over a field $k$. Then
the Hilbert scheme $X^{[n]}$ carries a natural
perfect obstruction theory whose virtual cycle
$$[X^{[n]}]^{\vir}\in CH_{2n}(X^{[n]})$$
whose pushforward to $X^{[n]}\times X^{[n]}$
is the Chern class $c_{2n}(R\homcal_{\pi}(\I_1,\I_2)[1])$.
\end{corollary}

Here $CH_{2n}(X^{[n]})$ is $\Z$ times the class of $X^{[n]}$,
and it follows from Gholampour-Thomas's construction
that the class
of the virtual cycle in Corollary \ref{diagonal}
is the integer 1 times the class of $X^{[n]}$.
Namely, the perfect obstruction theory on $X^{[n_1,n_2]}$
in Theorem \ref{nested} can be written as
$$\{ T(X^{[n_1]}\times X^{[n_2]})|_{X^{[n_1,n_2]}}
\arrow \Extcal^1_p(\I_1,\I_2)_0\}^*\arrow L_{X^{[n_1,n_2]}} $$
in the derived category of $X^{[n_1,n_2]}$.
Here $L_Y$ denotes the cotangent complex of $Y$, and at a point
$(I_1,I_2)$ in $X^{[n_1,n_2]}$, we define
$$\Extcal^1_p(\I_1,\I_2)_0=\coker(H^1(X,O)\arrow \Ext^1_X(I_1,I_2)),$$
where that map is associated to the given inclusion
$I_1\arrow I_2$. Here $\Extcal^1_p(\I_1,\I_2)_0$ is the tangent sheaf
to $X^{[n_1,n_2]}$. Therefore, the perfect obstruction
theory on $X^{[n]}$ in Corollary \ref{diagonal}
is 
$$\{ TX^{[n]}\oplus TX^{[n]}
\arrow \Extcal^1_p(\I_1,\I_2)_0\}^*\arrow L_{X^{[n]}}.$$
In this case, $\I_1$ and $\I_2$ are the same, and the map
is the sum of two isomorphisms $TX^{[n]}\arrow \Extcal^1_p(\I,\I)_0$.
So this perfect obstruction theory is equivalent to the
obvious one on the smooth variety $X^{[n]}$, and so the resulting
virtual cycle is 1 times the fundamental class of $X^{[n]}$.

\section{Torsion-freeness}

\begin{theorem}
\label{cohomology}
Let $X$ be a smooth complex projective surface.
If $H^*(X,\Z)$ is torsion-free,
then $H^*(X^{[n]},\Z)$ is torsion-free for every $n\geq 0$.
\end{theorem}

More generally, for any prime number $p$,
the same proof works $p$-locally. That is, if $H^*(X,\Z)$
has no $p$-torsion, then $H^*(X^{[n]},\Z)$ has no $p$-torsion
for every $n\geq 0$.

\begin{proof}
We follow Markman's argument on Poisson surfaces, with the extra input
of Corollary \ref{diagonal} \cite[proof of Theorem 1]{Markman}.
Bott periodicity says that topological $K$-theory is 2-periodic.
The differentials in the Atiyah-Hirzebruch spectral sequence
from $H^*(X,\Z)$ to $K^*(X)$
are always torsion \cite[section 2.4]{AH}.
Since $H^*(X,\Z)$ is torsion-free, the spectral sequence degenerates
at the $E_2$ page.
Also, the abelian group $H^*(X,\Z)$ is finitely generated
because $X$ is a closed manifold.
Therefore, $K^*(X)$ is a finitely generated free abelian group,
with $K^0(X)$ of rank $b_2(X)+2$ and $K^1(X)$ of rank $2b_1(X)$.
In this situation, the K\"unneth formula holds
for $K$-theory:
$$K^0(X\times Y)\cong \big[ K^0(X)\otimes_{\Z} K^0(Y)\big]
\oplus \big[ K^1(X)\otimes_{\Z} K^1(Y)\big] $$
for every finite CW-complex $Y$ \cite[Corollary 2.7.15]{Atiyah}.

Let $\{x_1,\ldots,x_m\}$ be a homogeneous basis
for $K^0(X)\oplus K^1(X)$. Write $u\mapsto u^{\vee}$
for the involution on $K^0$ of a space
that takes a vector
bundle to its dual, also known as the Adams operation
$\psi^{-1}$. (For a coherent sheaf $E$ on a smooth scheme $Y$, we interpret
$E^{\vee}$ to mean $\text{RHom}(E,O_Y)$ in the derived category of $Y$,
so it defines
the same operation on $K^0Y$.)
Consider the K\"unneth decomposition
$$\I=\sum_{i=1}^m x_i\otimes e_i$$
of the class of the universal ideal sheaf $\I$
in $K^0(X\times X^{[n]})$. Here the $e_i$ are some
(homogeneous) elements of $K^*(X^{[n]})$. Likewise, write
$$(\I)^{\vee}=\sum_{i=1}^m e_i'\otimes x_i$$
in $K^0(X^{[n]}\times X)$
for some (homogeneous) elements $e_i'\in K^*(X^{[n]})$.
Write $\chi\colon K^*(X)\arrow \Z$ for pushforward to a point
(which is defined because $X$ is a compact complex manifold).
For a coherent sheaf $E$, this is given by $\chi(E)=\sum_j (-1)^j h^j(X,E)$.

Write $\pi_{ij}$ for the projection from $X^{[n]}\times X\times X^{[n]}$
to the product of the $i$th and $j$th factors.
Then we have the equality
in $K^0(X^{[n]}\times X^{[n]})$:
$$(\pi_{13})_*[\pi_{12}^*(\I)^{\vee}\otimes^L \pi_{23}^*(\I)]=
\sum_{i=1}^m \sum_{j=1}^m (\pi_{13})_*(e_i'\otimes (x_ix_j)\otimes e_j).$$
For $x,y\in K^*(X)$, define $(x,y)=-\chi(xy)\in\Z$, the sign
being conventional for the Mukai pairing.
Using the projection formula, we have
$$(\pi_{13})_*[\pi_{12}^*(\I)^{\vee}\otimes^L \pi_{23}^*(\I)]=
-\sum_{i=1}^m \sum_{j=1}^m (x_i,x_j)e_i'\otimes e_j.$$

We need Markman's definition of the Chern classes
of an element of $K^1Y$, say for a finite CW complex $Y$
\cite[Definition 19]{Markman}. First, identify $K^1(Y)$
with $\K^0(\Sigma Y^{+})$, where $Y^{+}$ means the union of $Y$
with a disjoint base point, and $\K$ is the reduced $K$-theory
of a pointed space. For $u\in K^1(Y)$ and $i\geq 1/2$ congruent to 1/2
modulo $\Z$, define the Chern class $c_i(u)$ as the image
in $H^{2i}(Y,\Z)$ of $c_{i+1/2}(\widetilde{u})$, where
$\widetilde{u}$ is the corresponding element of $\K^0(\Sigma Y^{+})$,
and we identify $H^{2i}(Y,\Z)$ with $\widetilde{H}^{2i+1}(\Sigma Y^{+},\Z)$.
For $u,v\in K^1(Y)$, Markman showed that
the Chern classes of $uv\in K^0(Y)$
can be written as polynomials with integer coefficients in
the even-dimensional classes $c_i(u)c_j(v)$ \cite[Lemma 21]{Markman}.

By Corollary \ref{diagonal}, it follows that the diagonal
$\Delta\in H^{4n}(X^{[n]}\times X^{[n]},\Z)$ is given by
$$\Delta=c_{2n}\bigg( \sum_{i=1}^m \sum_{j=1}^m
(x_i,x_j)e_i'\otimes e_j\bigg).$$
By the formulas for the Chern classes of direct sums
and tensor products of elements of $K^0$, together with the result
above on Chern classes of the product of two elements of $K^1$,
it follows that $\Delta$ can be expressed as a sum
$$\Delta=\sum_{j\in J}\alpha_j\otimes \beta_j,$$
where each $\alpha_j$ and $\beta_j$ is a polynomial
with integer coefficients in the Chern classes
of $e_1,\ldots,e_m,e_1',\ldots,e_m'$.

Viewed as a correspondence, the diagonal acts 
as the identity on integral cohomology. That is, for any
element $u\in H^*(X^{[n]},\Z)$, we have
$$u=(p_1)_*(\Delta\cdot p_2^*(x)).$$
Combining this with the decomposition of the diagonal
above, we find that $u$ is a $\Z$-linear combination
of the elements $\alpha_j$:
$$u=\sum_{j\in J}\bigg( \int_{X^{[n]}}u\beta_j\bigg) \alpha_j.$$
If $u$ is torsion, then all the intersection numbers
$\int u\beta_j\in \Z$ are zero, and so $u=0$. That is,
$H^*(X^{[n]},\Z)$ is torsion-free, as we want.
\end{proof}

\section{Integral Chow motive}

Finally, we show that if the Chow motive with integral
coefficients of a smooth projective surface $X$ over a field $k$
is trivial
(a direct sum of Tate motives), then the same holds
for all Hilbert schemes $X^{[n]}$. The analogous statement
with rational coefficients is known, by de Cataldo
and Migliorini's general
description of the motive of $X^{[n]}$ with rational
coefficients \cite[Theorem 6.2.1]{dCMmotive}.

The Chow motive with integral coefficients is a direct sum
of Tate motives for every smooth complex projective rational surface,
but also for some Barlow surfaces,
which are of general type \cite[Proposition 1.9]{ACP},
\cite[Theorem 4.1]{Totaromotive}.

\begin{theorem}
\label{motive}
Let $X$ be a smooth projective surface over a field $k$. Let $R$ be
a PID of characteristic zero, meaning that
$\Z$ is a subring of $R$.
If the Chow motive of $X$ with coefficients in $R$
is a finite direct sum of Tate motives $R(a)$, then
the Hilbert scheme $X^{[n]}$ has the same property
for every $n\geq 0$.
\end{theorem}

\begin{proof}
By Gorchinsky and Orlov, since the Chow motive of $X$
with coefficients in $R$ is a finite direct sum of Tate motives
and $\Z$ is a subring of $R$,
the $K$-motive of $X$ with coefficients in $R$
is a finite direct sum of $K$-motives of points
\cite[Proposition 4.1]{GO}. It follows that the K\"unneth
formula holds for algebraic $K$-theory of products with $X$,
meaning that for every smooth projective variety $Y$, the product map
$$K_0(X)\otimes_{\Z}K_0(Y)\otimes_{\Z} R
\arrow K_0(X\times Y)\otimes_{\Z}R$$
is an isomorphism.

Given that, the proof of Theorem \ref{cohomology}
produces elements $e_i,e_i'$ in $K_0(X^{[n]})\otimes R$
using the K\"unneth formula on $X\times X^{[n]}$.
The argument then shows that the diagonal in the Chow group
$CH^{2n}(X^{[n]}\times X^{[n]})\otimes R$ is completely decomposable,
as a sum $\sum_j \alpha_j\otimes \beta_j$. Using that
$R$ is a PID, it follows that the Chow motive
of $X^{[n]}$ with coefficients in $R$ is a finite
direct sum of Tate motives $R(a)$
\cite[proof of Theorem 4.1]{Totaromotive}.
\end{proof}


\small \sc UCLA Mathematics Department, Box 951555,
Los Angeles, CA 90095-1555

totaro@math.ucla.edu
\end{document}